\documentclass[12pt]{amsart}

\def\urltilda{\kern -.15em\lower .7ex\hbox{\~{}}\kern .04em}

\usepackage[left=1.2in, bottom= 1.1in, right= 1.2in, top= 1.2in]{geometry} 
\usepackage{graphicx}
\usepackage{color}
\usepackage{amssymb}
\usepackage{ latexsym }
\usepackage{amsthm}
\usepackage{mathrsfs}
\usepackage{thm-restate}
\usepackage[colorlinks = true,
            linkcolor = red,
            urlcolor  = red,
            citecolor = red,
            anchorcolor = red]{hyperref}

\newtheorem{theorem}{Theorem}[section]
\newtheorem{lemma}[theorem]{Lemma}

\theoremstyle{definition}
\newtheorem{definition}[theorem]{Definition}

\newtheorem{cor}[theorem]{Corollary}

\newtheorem{claim}{Claim}

\newtheorem{proposition}[theorem]{Proposition}

\theoremstyle{project}

\theoremstyle{remark}

\numberwithin{equation}{section}

\def\urltilda{\kern -.15em\lower .7ex\hbox{\~{}}\kern .04em}

\newcommand{\bd}{\partial}
\newcommand{\set}[1]{\{ #1 \}}
\newcommand{\rmv}{\smallsetminus}

\newcommand{\wh}[1]{\widehat{ #1 }}

\begin{document}

\title{Hyperbolic Manifolds Containing High Topological Index Surfaces}

\author{Marion Campisi}
\address{San Jos\'{e} State University\\ San Jose, CA 95192}
\email{marion.campisi@sjsu.edu}

\author{Matt Rathbun}
\address{California State University, Fullerton \\
Fullerton, CA 92831}
\email{mrathbun@fullerton.edu}

\begin{abstract} If a graph is in bridge position in a 3-manifold so that the graph complement is irreducible and boundary irreducible, we generalize a result of Bachman and Schleimer to prove that the complexity of a surface properly embedded in the complement of the graph bounds the graph distance of the bridge surface. We use this result to construct, for any natural number $n$, a hyperbolic manifold containing a surface of topological index $n$.  
\end{abstract}

\maketitle
\date{}

\section{Introduction}

It has become increasingly common and useful to measure distances in complexes associated to surfaces between certain important sub-complexes associated with the surface embedded in a 3-manifold. These techniques provide a means to indicate the inherent complexity of links in a manifold, decomposing surfaces, or the manifold itself. In \cite{BacTITS3M} Bachman defined the topological index of a surface as a topological analogue of the index of an unstable minimal surface.  When the distance is small, the notion of topological index refines this distance, by looking at the \emph{homotopy type} of a certain sub-complex. 

In the same way that incompressible surfaces share important properties with strongly irreducible surfaces (distance $> 2$) despite being compressible, the topological index provides a degree of measurement of how similar irreducible, but weakly reducible (distance $= 1$) surfaces are to incompressible surfaces. In a series of papers \cite{BacNTMSI, BacNTMSII, BacNTMSIII}, Bachman has shown that surfaces with a well-defined topological index in a 3-manifold can be put into a sort of normal form with respect to a trianglulation of the manifold, generalizing the ideas of normal form introduced by Kneser \cite{KneGFDM} and almost normal form introduced by Rubinstein \cite{RubAAR3S}, and mirroring results about geometrically minimal surfaces due to Colding and Minicozzi \cite{ColMinSEMSFG3MI,ColMinSEMSFG3MII,ColMinSEMSFG3MIII,ColMinSEMSFG3MIV,ColMinSEMSFG3MV}.

Lee \cite{LeeOTMSHG} has shown that an irredubible manifold containing an incompressible surface contains topologically minimal surfaces of arbitrarily high genus, but has only shown that the topological index of such surfaces is at least two. In \cite{BacJohOEHITMS} Johnson and Bachman showed that surfaces of arbitrarily high index exist.   These surfaces are the lifts of Heegaard surfaces in an n-fold cover of a manifold obtained by gluing together boundary components of the complement of a link in $S^3$.  A by-product of their construction is that the resulting manifolds are toroidal.  

This leaves open the question of whether the much more ubiquitous class of hyperbolic manifolds can also contain high topological index surfaces. Here we construct certain hyperbolic manifolds containing such surfaces.  We generalize the construction in \cite{BacJohOEHITMS}  by gluing along the boundary components of the  complement of a graph in $S^3$ to show:  

\begin{restatable}{theorem}{TheoremMainTheorem}\label{theorem:MainTheorem}

There is a closed 3-manifold $M^1$, with an index 1 Heegaard surface $S$, such that for each $n$, the lift of $S$ to some $n$-fold cover $M^n$ of $M^1$ has topological index $n$.  Moreover, $M^n$ is hyperbolic for all $n$.  
\end{restatable}

In order to guarantee the hyperbolicity of $M^n$ we must rule out the existence of high Euler characteristic surfaces in the graph complement.  To that end, we define the \emph{graph distance}, $d_{\mathcal{G}}$, of graphs in $S^3$, an analogue of bridge distance of links.  In the spirit of Hartshorn \cite{HarHSHMHB} and Bachman-Schleimer \cite{BacSchlDBP} we show that the complexity of an essential surface is bounded below by the graph bridge distance:

\begin{restatable}{theorem}{TheoremTopologyBoundsDistanceGraphComplement}\label{theorem:TopologyBoundsDistanceGraphComplement}
Let $\Gamma$ be a graph in a closed, orientable 3-manifold $M$ which is in bridge position with respect to a Heegaard surface $B$, so that $M\rmv n(\Gamma)$ is irreducible and boundary irreducible. Let $S$ be a properly embedded, orientable, incompressible, boundary-incompressible, non-boundary parallel surface in $M \rmv n(\Gamma)$. Then $d_{\mathcal{G}}(B, \Gamma)$ is bounded above by $2 ( 2 g(S) + |\bd S| - 1)$.
\end{restatable}

In Section \ref{section:Definitions} we lay out the definitions of the various complexes and distances we will use, and prove Theorem \ref{theorem:TopologyBoundsDistanceGraphComplement}. In Section \ref{section:MainTheorem}, we prove Theorem \ref{theorem:MainTheorem}.

\section{Definitions}
\label{section:Definitions}

Given a link $\mathcal{L}\subset S^3$, a \emph{bridge sphere for $\mathcal{L}$} is a sphere, $B$, embedded in $S^3$, intersecting the link $\mathcal{L}$ transversely, and dividing $S^3$ into two $3$-balls, $V$ and $W$, so that there exist disks $D_V$ and $D_W$ properly embedded in $V$ and $W$, respectively, so that $\mathcal{L} \cap V \subset D_V$ and $\mathcal{L} \cap W \subset D_W$ are each a collection of arcs.

In \cite{GodBITC3S}, Goda introduced the notion of a bridge sphere for a spatial $\theta$-graph, and this was extended by Ozawa in \cite{OzaBPRSG}. A \emph{bridge sphere for a (spatial) graph $\Gamma$} is a sphere, $B$, embedded in $S^3$, instersecting $\Gamma$ transversely in the interior of edges, and dividing $S^3$ into two $3$-balls, $V$ and $W$, so that there exist disks $D_V$ and $D_W$ properly embedded in $V$ and $W$, respectively, so that $\Gamma \cap V \subset D_V$ and $\Gamma \cap W \subset D_W$ are each a collection of trees and/or arcs. 

If $B$ is a bridge sphere for a link $\mathcal{L}$, then a \emph{bridge disk} is a disk properly embedded in one of the components of $\overline{(S^3 \rmv n(\mathcal{L})) \rmv B)}$, whose boundary consists of exactly two arcs, meeting at their endpoints, with one arc essential in $B \rmv n(\mathcal{L})$, and the other essential in $\bd n(\mathcal{L})$. We refer to the arc in the boundary of the disk that is contained in $B$ as a \emph{bridge arc}. Similarly, if $B$ is a bridge sphere for a graph $\Gamma$, then a \emph{graph-bridge disk} is a disk properly embedded in one of the components of $\overline{(S^3 \rmv n(\Gamma)) \rmv B)}$, whose boundary consists of exactly two arcs, meeting at their endpoints, with one arc essential in $B \rmv n(\Gamma)$, and the other essential in $\bd n(\Gamma)$. We refer to the arc in the boundary of the disk that is contained in $B$ as a \emph{graph-bridge arc}.

\begin{definition}
The \emph{curve complex} for a surface $B$ with (possibly empty) boundary is the complex with vertices corresponding to the isotopy classes of essential simple closed curves in $B$, so that a collection of vertices defines a simplex if representatives of the corresponding isotopy classes can be chosen to be pairwise disjoint. We will denote the curve complex for a surface $B$ by $\mathcal{C}(B)$.
\end{definition}

\begin{definition}
The \emph{arc and curve complex} for a  surface $B'$ with boundary is the complex with vertices corresponding to the (free) isotopy classes of essential simple closed curves and properly embedded arcs in $B'$. A collection of vertices defines a simplex if representatives of the corresponding isotopy classes can be chosen to be pairwise disjoint. We will denote the arc and curve complex for a surface $B'$ by $\mathcal{AC}(B')$.  
\end{definition}

If $B$ is a surface embedded in a manifold, and a 1-dimensional complex intersects $B$ transversely, we will refer to the surface obtained by removing a neighborhood of the 1-complex by $B'$. We will often refer to $\mathcal{C}(B')$ simply by $\mathcal{C}(B)$, and $\mathcal{AC}(B')$ simply by $\mathcal{AC}(B)$.

\begin{definition}
Let $B$ be a surface with at least two distinct, essential curves.  Given two collections $X$ and $Y$ of vertices in the complex $\mathcal{C}(B)$ (resp., $\mathcal{AC}(B)$), the distance between $X$ and $Y$ , denoted $d_{\mathcal{C}(B)}(X, Y)$ (resp., $d_{\mathcal{AC}(B)}(X, Y)$), is the minimal number of edges in any path in $\mathcal{C}(B)$ (resp., $\mathcal{AC}(B)$) from a vertex in $X$ to a vertex in $Y$. When the surface is understood, we often just write $d_{\mathcal{C}}$ (resp., $d_{\mathcal{AC}}$).
\end{definition}

We will be working with four subtly different but closely related sub-complexes, and some associated notions of distance.

\begin{definition} Let $B$ be a properly embedded surface separating a manifold $M$ into two components, $V$ and $W$. Define the \emph{disk set of $V$} (resp., \emph{$W$}), denoted $\mathcal{D}_V\subset \mathcal{C}(B)$, (resp. $\mathcal{D}_W\subset \mathcal{C}(B)$), as the set of all vertices corresponding to essential simple closed curves in $B$ that bound embedded disks in $V$ (resp., $W$). Define the \emph{disk set of $B$}, denoted $\mathcal{D}_B$, as the set of all vertices corresponding to essential simple closed curves in $B$ that bound embedded disks in $M$.
\end{definition}

\begin{definition}
Let $B$ be a bridge sphere for a link $\mathcal{L}$, bounding $3$-balls $V$ and $W$, with at least $6$ marked points corresponding to the transverse intersections of $\mathcal{L}$ with $B$. The \emph{distance of the bridge surface}, denoted $d_{\mathcal{C}}(B,\mathcal{L})$, is $d_{\mathcal{C}(B')}(\mathcal{D}_V, \mathcal{D}_W)$, the distance in the curve complex of $B'$ between $\mathcal{D}_V$ and $\mathcal{D}_W$.
\end{definition}

The fundamental building block in our construction will be the exterior of a graph that is highly complex as viewed from the arc and curve complex. The existence of such a block will follow from a result of Blair, Tomova, and Yoshizawa. It is a special case of Corollary 5.3 from \cite{BlaTomYosHDBS}.

\begin{theorem}[\cite{BlaTomYosHDBS}]
\label{theorem:ConstructingHighDistanceLinks}
Given non-negative integers $b$ and $d$, with $b \geq 3$, there exists a $2$-component link $\mathcal{L}$ in $S^3$, and a bridge sphere $B$ for $\mathcal{L}$ so that $\mathcal{L}$ is $b$-bridge with respect to $B$ and $d_{\mathcal{C}}(B, \mathcal{L}) \geq d$.  
\end{theorem}

\begin{definition} Let $B$ be a bridge sphere for a link $\mathcal{L}$, bounding $3$-balls $V$ and $W$.  Define the \emph{bridge disk set} of $V$ (resp., $W$), denoted $\mathcal{BD}_V\subset \mathcal{AC}(B)$ (resp., $\mathcal{BD}_W$), as the set of all vertices either corresponding to essential simple closed curves in $B'$ that bound embedded disks in $V \rmv \mathcal{L}$ (resp., $W \rmv \mathcal{L}$), or corresponding to bridge arcs in $B'$.
\end{definition}

\begin{definition}
Let $B$ be a bridge sphere for a link $\mathcal{L}$, bounding $3$-balls $V$ and $W$. The \emph{bridge distance of the bridge surface $B$}, denoted $d_{\mathcal{BD}}(B,\mathcal{L})$ is $d_{\mathcal{AC}(B')}(\mathcal{BD}_V, \mathcal{BD}_W)$, the distance in the arc and curve complex of $B'$ between $\mathcal{BD}_V$ and $\mathcal{BD}_W$.  
\end{definition}

\begin{lemma}[\cite{BlaCamJohTayTomECSK}, Lemma 2]
\label{lemma:DifferenceDistanceCurveComplexArcCurveComplex}
If $B$ is a bridge surface which is not a sphere with four or fewer punctures, then $d_{\mathcal{BD}}(B,\mathcal{L})\leq d_{\mathcal{C}}(B,\mathcal{L})\leq 2d_{\mathcal{BD}}(B,\mathcal{L})$.
\end{lemma}

\begin{definition}
Let $B$ be a bridge sphere for graph $\Gamma$, bounding 3-balls $V$ and $W$. The \emph{graph disk set} of $V$ (resp., $W$) denoted $\mathcal{GD}_V\subset \mathcal{AC}(B)$ (resp., $\mathcal{GD}_W\subset \mathcal{AC}(B)$), is the set of all vertices either corresponding to essential simple closed curves in $B \rmv n(\Gamma)$ that bound embedded disks in $V \rmv n(\Gamma)$ (resp., $W \rmv n(\Gamma)$), or corresponding to graph-bridge arcs in $B \rmv n(\Gamma)$.
\end{definition}

\begin{definition}
Let $B$ be a bridge sphere for graph $\Gamma$.   The \emph{graph distance of the bridge surface}, denoted $d_{\mathcal{G}}(B, \Gamma)$ is $d_{\mathcal{AC}(B')}(\mathcal{GD}_V, \mathcal{GD}_W)$, the distance in the arc and curve complex of $B' = B \rmv n(\Gamma)$ between $\mathcal{GD}_V$ and $\mathcal{GD}_W$.  
\end{definition}

\begin{lemma}
\label{proposition:disjointdisk}
Let $\mathcal{L}$ be a link in bridge position with respect to bridge sphere $B$, bounding 3-balls $V$ and $W$, and let $\Gamma_{\mathcal{L}}$ be a graph in bridge position with respect to $B$ formed by adding edges to $\mathcal{L}$ in $V$ that are simultaneously parallel into $B$ in the complement of $\mathcal{L}$, and so that $\Gamma_{\mathcal{L}}\cap V$ has at least two components.  

If $D\subset (V \rmv n(\Gamma_{\mathcal{L}}))$ is a graph-bridge disk for $\Gamma_{\mathcal{L}}$, then there is a bridge disk $D'$ for $\mathcal{L}$ in $(V \rmv n(\mathcal{L}))$ which is disjoint from $D$.  
\end{lemma}

\begin{proof}
Let $\Gamma_{1}$, \dots, $\Gamma_{\ell}$ be the connected components  $\Gamma_{\mathcal{L}}\cap V$, and let $\Gamma_i$ be the component of $\Gamma_{\mathcal{L}} \cap V$ to which $D$ is incident.  

Over all bridge disks  $E\subset V$  for $\mathcal{L}$ disjoint from $ \Gamma_i$, choose one which minimizes $|D\cap E|$.  Suppose the intersection is non-empty.  Any loops of intersection can be removed because $(V \rmv n(\Gamma))$ is a handlebody and therefore irreducible.   Any points of intersection between $\partial D$ and $\partial E$ are contained in $\partial D \cap B$ and $\partial E\cap B$.    Choose an arc $\gamma$ of $|D\cap E|$.  The arc $\gamma$ cuts $D$ into two disks $D_{\gamma_1}$ and $D_{\gamma_2}$.  For one of $i=1$ or $2$, $\partial D_{\gamma_i} \cap \partial D$ is contained in $B$.  Call that disk $D_{\gamma}$.   Consider an  arc $\alpha$ of $|D\cap E|$ outermost in $D_{\gamma}$.  If the interior of $D_{\gamma}$ is disjoint from $E$ then take $\alpha$ to be $\gamma$.  The arc $\alpha$ cuts off a disk $D_{\alpha}$ from $D_{\gamma}$ and cuts $E$ into two disks $E_1$ and $E_2$ only one of whose (say $E_2$) boundary is incident to $\mathcal{L}$.  The disk $E_2\cup D_{\alpha}=E'$ is a bridge disk for $\mathcal{L}$ and intersects $D$ fewer times than $E$, contradicting  the minimality of $|D\cap E|$.  
\end{proof}

The above implies that the distance in the arc and curve complex of $B \rmv n(\Gamma)$ between $\mathcal{GD}_V$ and $\mathcal{BD}_V$ is less than or equal to one.  

\begin{cor} \label{cor:ACtoG} Let $\mathcal{L}$ and $\Gamma_{\mathcal{L}}$ be as above.  
Then $d_\mathcal{\mathcal{BD}}(B,\mathcal{L}) \leq 1 + d_\mathcal{G}(B,\Gamma_{\mathcal{L}})$.  
\end{cor}

\begin{proof}
Since $W \rmv n(\Gamma)$ contains no graph-bridge disks, $\mathcal{GD}_W=\mathcal{BD}_W$.  Thus $d_\mathcal{G}(B,\Gamma_\mathcal{L})=d_{\mathcal{AC}}(\mathcal{GD}_V, \mathcal{BD}_W)$.  Lemma \ref{proposition:disjointdisk} shows that $d_{\mathcal{AC}}(\mathcal{GD}_V, \mathcal{BD}_V)\leq 1$, and so by the triangle inequality we have that $d_\mathcal{BD}(B,\mathcal{L}) \leq 1 + d_\mathcal{G}(B,\Gamma_\mathcal{L})$.
\end{proof}

In \cite{HarHSHMHB}, Hartshorn proved that an essential closed surface in a 3-manifold creates an upper bound on the possible distances of Heegaard splittings of that manifold in terms of the genus of the essential surface. 

\begin{theorem}[Hartshorn, Theorem 1.2 of \cite{HarHSHMHB}] \label{theorem:Hartshorn} Let $M$ be a Haken 3-manifold containing an incompressible surface of genus $g$. Then any Heegaard splitting of $M$ has distance at most $2g$.
\end{theorem}

This idea has been generalized in numerous ways, including by Bachman and Schleimer, who show in \cite{BacSchlDBP} that the distance of a bridge Heegaard surface in a knot complement is bounded by twice the genus plus the number of boundary components of an essential properly embedded surface.

\begin{theorem}[Bachman-Schleimer, Theorem 5.1 of \cite{BacSchlDBP}] \label{theorem:Bachman-Schleimer} Let $K$ be a knot in a closed, orientable 3-manifold $M$ which is in bridge position with respect to a Heegaard surface $B$. Let $S$ be a properly embedded, orientable, essential surface in $M \rmv n(K)$. Then the distance of $K$ with respect to $B$ is bounded above by twice the genus of $S$ plus $| \bd S|$.
\end{theorem}

We will need a yet more general version, since we will be concerned with surfaces properly embedded in \emph{graph} complements.
 
The essence of both results is that the distance of a bridge or Heegaard surface is bounded above in terms of the \emph{complexity} of an essential properly embedded surface. We will generalize this result to link and graph complements, with the additional benefit of avoiding many of the technical details of \cite{BacSchlDBP} necessary to treat the boundary components. Unfortunately, our bound will be worse than that obtained by Bachman and Schleimer, though it will be sufficient for many applications of this type of bound (\emph{e.g.}, \cite{MosSHSHDK}, \cite{DuQiuNUUHSA3M}, \cite{OhsSakSMCGRHSBD}, \cite{BacSDHSSC3M}, and \cite{NamBHDIFMCG}). We note also that our proof requires a minimal starting position similar to that used by Hartshorn, an assumption the Bachman-Schleimer method was able to avoid.

We now prove 
the following.
\TheoremTopologyBoundsDistanceGraphComplement*

\begin{proof}[Proof of Theorem \ref{theorem:TopologyBoundsDistanceGraphComplement}]
In the case that $S$ is closed, we note that the proofs of both Theorem \ref{theorem:Bachman-Schleimer} and Theorem \ref{theorem:Hartshorn} apply to closed surfaces in manifolds with boundary as long as the manifold is irreducible. In the case that $\partial S\neq \emptyset$ we will double $M \rmv n(\Gamma)$ along $\bd n(\Gamma)$ to obtain a closed surface and show that the surface can be made to fulfill all the hypotheses necessary to use the machinery in the proof of Theorem \ref{theorem:Hartshorn} to obtain the bound on distance.

First, isotope $S$ to intersect $B$ minimally, among all isotopy representatives of $S$.  Let $V$ and $W$ be the handlebodies on either side of $B$.   Double $M \rmv n(\Gamma)$ along $\bd n(\Gamma)$, and call the resulting manifold $\wh{M}$.  Let the doubles of $S$, $B$, $V$ and $W$ be $\wh{S}$, $\wh{B}$, $\wh{V}$ and $\wh{W}$, respectively, and let $G$ be $\bd n(\Gamma)$ in $\wh{M}$, with respective copies $M_i$, $S_i$, $B_i$, $V_i$ and $W_i$ for $i=1, 2$.

Note that $\wh{B}$ is a Heegaard surface for $\wh{M}$. (The proof of this is very similar to the proof of Proposition \ref{proposition:HeegaardSurface} below.) Also, note that since $S$ is incompressible and $\bd$-incompressible in $M \rmv n(\Gamma)$, $\wh{S}$ is an incompressible closed surface in $\wh{M}$ and since $\bd n(\Gamma)$ was incompressible in $M \rmv n(\Gamma)$, $G$ is incompressible in $\wh{M}$.  

\begin{claim}
\label{claim:Incompressible}
Each of $\wh{S}\cap \wh{V}$ and $\wh{S}\cap \wh{W}$ are incompressible.  
\end{claim}
\begin{proof}
If, say, $\wh{S}\cap \wh{V}$ had a compressing disk $D$, then since $\wh{S}$ is incompressible in $\wh{M}$, there would have to be a disk $D'$ in $\wh{S}$ with $\bd D'=\bd D$, and $D'\cap \wh{B}\neq \emptyset$.  We may choose $D$ to be a compressing disk which intersects $G$ minimally.  Further, since $G$ is incompressible, we may choose $D$ to intersect $G$ only in arcs, if at all.   But $\wh{M}$ is irreducible, so $D\cup D'$ bounds a ball and we may isotope $\wh{S}$ across this ball from $D'$ to $D$, lowering the number of intersections between $\wh{S}$ and $\wh{B}$.  

If $D'\cap G = \emptyset$, then this can be viewed as an isotopy of $S$ in $M \rmv n(\Gamma)$ which reduces the number of intersections between $S$ and $B$, a contradiction.  

If $D'\cap G \neq \emptyset$ we still arrive at a contradiction.  Consider a loop, $\ell$, of intersection in $(D\cup D')\cap G$, innermost in $D\cup D'$.  Since $D\cap G$ only contains arcs, $\ell$ consists of two arcs, $\alpha$ and $\alpha '$ in $D$ and $D'$ respectively.  Thus $\ell$ bounds a disk $D_{\ell}$ in $G$,  $\alpha$ cuts off a subdisk $D_{\alpha}$ of $D$ and $\alpha'$ cuts off a subdisk $D_{\alpha'}$ of $D'$, both of which are in either $M_1$ or $M_2$, say $M_1$.  Now we have an isotopy of $S_1$ from $D_{\alpha}\cup D_{\alpha'}$ to $D_{\ell}$

Independent of whether $D_{\alpha'}$ intersected $B$, we could have chosen $D$ to have fewer intersections with $G$, contradicting our choice of $D$ to minimize intersections.  
\end{proof}

\begin{claim}
\label{claim:Essential}
Every intersection of $\wh{S}$ with $\wh{B}$ is essential in $\wh{B}$.  
\end{claim}
\begin{proof}
Curves of intersection in $\wh{S}\cap \wh{B}$ which are inessential in both surfaces would either give rise to a reduction in $|S\cap B|$ or could have come from the doubling of arcs in $S\cap B$ which would  give rise to a reduction in $|S\cap B|$ in a fashion similar to the previous claim.  
\end{proof}

\begin{claim}
\label{claim:NoBoundaryParallelAnnuli}
There are no $\partial$-parallel annular components of $\wh{S}\cap \wh{W}$ or $\wh{S}\cap \wh{V}$.
\end{claim}
\begin{proof}
Any such component disjoint from $G$ would have been eliminated when   $|S\cap B|$ was minimized.  The intersection of any such component intersecting $G$ with $M_1$ would be a $\partial$-parallel disk which also would have been eliminated when $|S\cap B|$ was minimized.
\end{proof}

Now we have satisfied all the hypotheses to obtain the sequence of isotopic copies of $\wh{S}$ described in Lemmas 4.4 and 4.5 of \cite{HarHSHMHB}.  Depending on whether either of  $\wh{S}\cap \wh{V}$ or $\wh{S}\cap \wh{W}$ contain disk components or not, we apply either Lemma 4.4 or 4.5, respectively, of \cite{HarHSHMHB} to obtain a sequence of compressions of $\wh{S}$ which give rise to a path in  $\mathcal{AC}(\wh{S})$.  A priori, this path would not restrict to a path in $\mathcal{AC}(S)$, but the following Claim shows that we can choose the compressions to be symmetric across  $G$, and so each compression will correspond to an edge in $\mathcal{AC}(S)$.  

\begin{claim}
If there exists an elementary $\partial$-compression of $\wh{S}$ in $\wh{V}$ (resp. $\wh{W}$), then there exists an elementary compression of $\wh{S}$ in $\wh{V}$ (resp. $\wh{W}$) which is symmetric across $G$ in the sense that either
\begin{enumerate}
\item the $\partial$-compressing disk $D_1$  is disjoint from $G$ in $M_1$, and there is a corresponding $\partial$-compressing disk $D_2$ in $M_2$,
or
\item the $\partial$-compression is along a disk that is symmetric across $G$.  
\end{enumerate}

\end{claim}

\begin{proof}
Let $D$ be an elementary $\partial$-compression disk for, say, $\wh{S}\cap\wh{V}$ chosen to minimize $|D\cap G|$.  We may restrict attention to such disks with $|D\cap G|>0$.  

First, we observe that $D\cap G$ cannot contain any loops of intersection, for a loop of $D\cap G$ innermost in $D$ bounds a sub-disk of $D$ which would either give rise to a compression for $G$ or would provide a means of isotoping $D$ so as to lower $|D\cap G|$.  Thus, $D\cap G$ consists only of arcs.   These arcs are either 
\begin{itemize}
\item vertical arcs: with one endpoint on each of $\wh{S}$ and $\wh{B}$,
\item $\wh{S}$-arcs: with both endpoints on $\wh{S}$, or
\item $\wh{B}$-arcs: with both endpoints on $\wh{B}$.
\end{itemize}

Consider an $\wh{S}$-arc of $D\cap G$, outermost in $D$, cutting off sub-disk $D'$ from $D$, with boundary consisting of $\sigma$ in $\wh{S}$ and $\gamma$ in $G$. Without loss of generality, assume $D' \subset M_1$. If $\sigma$ is essential in $\wh{S}\cap M_1$, then $D'$ is a boundary compression disk for $S$ in $M$, which is impossible.  If $\sigma$ is inessential in $\wh{S}\cap M_1$, then it must co-bound a disk $E$ in $\wh{S}\cap M_1$ together with an arc $\sigma ' \subseteq \partial (\wh{S}\cap M_1)$.  The curve $\gamma \cup \sigma'$ cannot be essential in $G$, else $D'\cup E$ would be a compressing disk for $G$.  Thus, $\gamma \cup \sigma'$ bounds a disk, $F\subseteq G$.   Now $F\cup D'\cup E$ is a sphere bounding a ball in $M_1$, so $D\cup E$ is isotopic to $F$, and replacing $D'$ with $F$ results in an elementary boundary compressing disk for $\wh{S}\cap V$ with fewer  intersections with $G$ than $D$.  Thus we may assume that $D\cap G$ contains no $\wh{S}$-arcs.  

Now consider a sub-disk $D'$ of $D$ which is cut off by all the arcs of $D\cap G$ and whose boundary consists of no more than one vertical arc.  With out loss of generality, assume $D'\subseteq M_1$.  Suppose $\partial D'$ has $\wh{B}$-arcs, $\beta_1, \beta_2, \dots, \beta_k$.  Then all the $\beta_i$ are disjoint arcs on $G$.  If any of them are inessential in  $G\cap \wh{V}$ then they bound disks $B_i\subseteq G\cap V_1$.  If any of the $\beta_i$ are essential in $G\cap \wh{V}$, then they bound disks $B_i\subseteq V_1$ that are bridges disks for $n(\Gamma)$ in $V_1$.  In either case, $D' \cup \left( \bigcup_{i=1}^k B_i \right)$ results in a boundary compressing disk for $S\cap \wh{V}$ with fewer intersections with $G$ than $D$.  This boundary compressing disk is still elementary as the arc in $\wh{S}$ remains unchanged.  Thus, we may assume that $D\cap G$ consists solely of vertical arcs.  

Let $\gamma$ be an arc of $D\cap G$ outermost in $D$, cutting off a sub-disk $D_1$ from $D$.  Without loss of generality, $D_1\subseteq M_1$.  The boundary of $ D_1$ consists of three arcs; $\gamma \subseteq G$, $\sigma_1\subseteq S_1$ and $\beta_1\subseteq B_1$.  By symmetry, there exists disk $D_2\subseteq M_2$ in $M_2$, so that $D_1\cup D_2$ is a disk in $\wh{V}$ with boundary consisting of arcs $\sigma=\sigma_1\cup \sigma_2\subseteq \wh{S}$ and  $\beta=\beta_1\cup \beta_2\subseteq \wh{B}$, intersecting $G$ in exactly one arc, $\gamma$.  Finally, we must show that $\sigma$ is a ``strongly essential" arc in $\wh{S}\cap \wh{V}$.

If $\sigma$ is not strongly essential then  it is either the meridian of a boundary parallel annulus of $\wh{S}\cap \wh{V}$ which is 

not possible since $\sigma_1$ was a sub-arc of the original elementary compression disk $D$, or $\sigma$ is inessential in $\wh{S}\cap \wh{V}$.  If $\sigma$ is inessential  then it would co-bound a disk $E$ in $\wh{S}$ together with an arc $\sigma'\subseteq \wh{S}\cap\wh{B}$.  This disk provides an isotopy in $\wh{S}$ of  $\sigma_1$ to $\sigma_2$.  

If the disk $D'=D\rmv D_1$ only intersects $D_2$ in $\gamma$ then $D'\cup D_2$ is a  compressing disk for $\wh{S}\cap \wh{V}$ with fewer arcs of intersection with $G$, as the disk can be isotoped away from $\gamma$.  This disk is still an elementary compressing disk because $\sigma_1$ is isotopic to $\sigma_2$, and so contradicts our original choice of $D$.

Thus, $\sigma$ is strongly essential in $\wh{S}\cap \wh{V}$, and $D_1 \cup D_2$ is a new compressing disk for $\wh{S} \cap \wh{V}$ that is symmetric across $G$.
\end{proof}

We may, thus, proceed exactly as in Theorem \ref{theorem:Hartshorn}. Each elementary boundary compression of $\wh{S}$ towards either of $\wh{V}$ or $\wh{W}$ can be performed in a symmetric way, demonstrating a path from $\mathcal{D}_{\wh{V}}$ to $\mathcal{D}_{\wh{W}}$ in $\mathcal{C}(\wh{S})$ of length no greater than twice the genus of $\wh{S}$, which is $2(g(S) + |\bd S| - 1)$.

Each time a boundary compression for $\wh{S}$ corresponds to a pair of curves $\wh{c_i}$ and $\wh{c_{i+1}}$ in $S_1$ that contribute an edge in a path in $\mathcal{C}(\wh{S})$ from $\mathcal{D}_{\wh{V}}$ to $\mathcal{D}_{\wh{W}}$, there is immediately a pair of curves $\wh{c_{i+2}}$ and $\wh{c_{i+3}}$ in $S_2$ also contributing an edge in a path from $\mathcal{D}_V$ to $\mathcal{D}_W$, and this pair of paths corresponds to a single pair of curves $c_i$ and $c_{i+1}$ in $S$ contributing a single edge in $\mathcal{AC}(S)$. Each time a boundary compression for $\wh{S}$ corresponds to a pair of curves intersecting $G$ that contributes an edge in a path in $\mathcal{C}(\wh{S})$ from $\mathcal{D}_{\wh{V}}$ to $\mathcal{D}_{\wh{W}}$, the restriction of these curves to $S_1$ is a pair of arcs contributing an edge in $\mathcal{AC}(S)$. 

Further, since the boundary compressions (and elimination of boundary parallel annuli) are all being performed symmetrically, the resulting disks $D_{\wh{V}} \in \mathcal{D}_{\wh{V}}$ from $\wh{S} \cap \wh{V}$ and $D_{\wh{W}} \in \mathcal{D}_{\wh{W}}$ from $\wh{S} \cap \wh{W}$ are symmetric. That is, either $D_{\wh{V}}$ (resp., $D_{\wh{W}}$) is disjoint from $G$, so that we may assume that it sits in $V_1$ (resp., $W_1$), or it is symmetric across $G$ so that $D_{\wh{V}} \cap M_1$ (resp., $D_{\wh{W}} \cap M_1)$ is a graph bridge disk for $\Gamma$ in $M$. In either case, this demonstrates a path in $\mathcal{AC}(S)$ from $\mathcal{DG}_V$ to $\mathcal{DG}_W$ of length no greater than $2(g(S) + |\bd S| - 1)$.
\end{proof}

\section{Theorem \ref{theorem:MainTheorem}}
\label{section:MainTheorem}

In \cite{BacTITS3M} Bachman defined the topological index of a surface. In contrast to the distances between sub-complexes each corresponding to some disks discussed in Section \ref{section:Definitions}, he exploits the homotopy type of the complex of all disks.

\begin{definition} The surface $B$ is said to be \emph{topologically minimal} if either $\mathcal{D}_B$ is empty, or if there exists an $n \in \mathbb{N}$ so that $\pi_n(\mathcal{D}_B) \neq 0$. If a surface $B$ is topologically minimal, then the \emph{topological index} is defined to be the smallest $n \in \mathbb{N}$ so that $\pi_{n-1}(\mathcal{D}_B) \neq 0$, or $0$ if $\mathcal{D}_B$ is empty. 
\end{definition}

In \cite{BacJohOEHITMS} Johnson and Bachman showed that surfaces of arbitrarily high index exist, but the manifolds they construct all contain essential tori. We prove an analogue  of this.

\TheoremMainTheorem*

\subsection{The construction}
Let $n$ be a positive integer. We will construct a hyperbolic manifold containing a Heegaard surface of topological index $n$.

Using the machinery in Theorem \ref{theorem:ConstructingHighDistanceLinks}, let $\mathcal{L}$ be a $(0, 4)$-link in $S^3$ with two components, $L$ and $K$, with bridge sphere $B$ of distance at least $32n + 7$. Let $V$ and $W$ be the two 3-balls bounded by $B$. Since $\mathcal{L}$ is in bridge position, there exist disks $D_V$ and $D_W$ properly embedded in $V$ and $W$, respectively, with $(\mathcal{L} \cap V ) \subset D_V$, and $(\mathcal{L} \cap W ) \subset D_W$. By modifying $D_V$ if necessary, we can find two arcs $\tau_L$ and $\tau_K$ in $B$ such that

\begin{enumerate}
\item $\tau_L \cup \tau_K \subset D_V$,
\item $\tau_L \cap \tau_K = \emptyset$,
\item $\tau_L \cap \mathcal{L} = \bd \tau_L \subset L$ and $\tau_K \cap \mathcal{L} = \bd \tau_K \subset K$, 
\item each of $\tau_K$ and $\tau_L$ have endpoints on different components of $\mathcal{L} \cap V$.
\end{enumerate}
 Let $L' = L \cup \tau_L$, let $G_{L} = \bd n(L')$, let $K' = K \cup \tau_K$, let $G_{K} = \bd n(K')$, and let $\Gamma=\mathcal{L} \cup \tau_L \cup \tau_K = L' \cup K'$. Observe that $\Gamma$ is a graph in bridge position with respect to $B$. Let $M' = \overline{S^3 \rmv n(\Gamma)}$, let $V'=\overline{V\rmv n(\Gamma)}$, and let $W' = \overline{W \rmv n(\Gamma)} = \overline{W \rmv n(\mathcal{L})}$, and $B' = B \rmv n(\Gamma)=B \rmv n(\mathcal{L})$.

For each $i = 1, 2, \dots, n$, let $M'_i$ be homeomorphic to $M'$, along with homeomorphic copies $\mathcal{L}_i$ of $\mathcal{L}$, $(G_{L})_i$ of $G_{L}$, $(G_{K})_i$ of $G_{K}$, and $B'_i$ of $B'$.

Then, for each $i = 1, 2, \dots, (n-1)$, identify $(G_{K})_i$ with $(G_{L})_{i+1}$ and identify $(G_{K})_n$ with $(G_{L})_1$, all via the same homeomorphism. Call the resulting closed $3$-manifold $M^n$. Observe that the union of the $B'_i$ is a closed surface that we will call $B^n$. We will show that $B^n$ is a Heegaard surface for $M^n$, that $B^n$ has high topological index, and that $M^n$ is hyperbolic.

\begin{proposition}
\label{proposition:HeegaardSurface}
For each $n$, the surface $B^n\subset M^n$ is a genus $3n + 1 $ Heegaard surface.  
\end{proposition}

\begin{proof}
That the genus of $B^n$ is $3n + 1$ can be verified by an Euler characteristic count. It suffices, then, to verify that the complement of $B^n$ is two handlebodies, $V^n$ and $W^n$. 

Since $\Gamma$ was in bridge position with respect to $B$,  there are disks $D_V$ and $D_W$ properly embedded in $V$ and $W$, respectively, so that $\Gamma \cap V \subset D_V$ and $\Gamma \cap W \subset D_W$. Then $D_V$ and $D_W$ cut along $\Gamma$ is a collection of sub-disks. 

The result of cutting $V \rmv n(\Gamma)$ along all these sub-disks of $D_V$ is a pair of 3-balls, each with two sub-disks, $D_1^+$ and $D_2^+$, of $n(\Gamma)$ contained in the boundary. Each identification of $(G_{K})_i$ with $(G_{L})_{i+1}$ (indices mod $n$) glues pairs of these sub-disks along arcs, resulting in disks in $V^n$, and further cutting along $(n-1)$ copies of each of $D_1^+$ and $D_2^+$ results in a collection of $3$-balls, showing that $V^n$ is a handlebody.

Similarly, the result of cutting $W \rmv n(\Gamma)$ along all of the sub-disks of $D_W$ is a pair of 3-balls, each with four sub-disks of $n(\Gamma)$ contained in the boundary, $D_1^-, D_2^-, D_3^-$, and $D_4^-$. Each identification of $(G_{K})_i$ with $(G_{L})_{i+1}$ (indices mod $n$) glues pairs of these sub-disks along arcs, resulting in disks in $W^n$, and further cutting along $(n-1)$ copies of each of $D_1^-, D_2^-, D_3^-$, and $D_4^-$ results in a collection of $3$-balls, showing that $W^n$ is a handlebody.
\end{proof}

\subsection{Bounding from above}

\begin{proposition}
The surface $S^n$ has topological index at most $n$.
\end{proposition}

\begin{proof} Our proof will follow almost exactly as the proof of Proposition 5 from \cite{BacJohOEHITMS}. 
In each copy $M_i'$ of the manifold $M'$, we have the surface $B'_i$, a copy of $B'$, dividing the manifold into $V_i'$ and $W_i'$, copies of $V'$ and $W'$. Observe that in each $V'_i$, there is exactly one essential disk, $D_i^+$ with boundary contained in $B_i'$, just as in \cite{BacJohOEHITMS}. However, in each $W_i'$, there are several essential disks with boundary contained in $B_i'$. We will call this collection of disks $\mathscr{D}^-_i$.  From each $\mathscr{D}^-_i$, choose a single representative $D^-_i$.

Define the sub-complex, $P$, of $\mathcal{D}_M$ spanned by the vertices corresponding to $\bigcup_i \set{D_i^+, D_i^-}$, which is homeomorphic to an $(n-1)$-sphere. Then, define a map $F: \mathcal{D}_M \to P$ by the identity on $P$, and by sending a vertex corresponding to a disk $D \not \in \bigcup_i \set{D_i^+, D_i^-}$ to the vertex corresponding to $D_j^+$ or $D_j^-$, where either $D \in \mathscr{D}_j^-$, or $j$ is the smallest index for which an essential outermost sub-disk of $D \rmv (\bigcup_i G_i)$ is contained in $V'_j$ or $W'_j$, respectively. 

Just as in \cite{BacJohOEHITMS}, we claim that this map $F$ is a simplicial map that fixes each vertex of $P$. To see this, consider any two disks $D_1$ and $D_2$ connected by an edge in $\mathcal{D}_M$ (so that the disks are realized disjointly in $M$). Observe that by our construction of $M'$ and Corollary \ref{cor:ACtoG}, any disk contained in $V'_j$ must intersect any disk contained in $W'_j$ (whether either disk is a bridge disk, a graph-bridge disk, or the boundary is contained in $B'_j$). So, if $D_i^{\pm} = F(D_1) \neq F(D_2) = D_j^{\pm}$, then $i \neq j$, and $F(D_1)$ is joined to $F(D_2)$ in $P$. Thus, $F$ is a retraction onto the $(n-1)$-sphere, $P$, showing that $\pi_{n-1}(\mathcal{D}_M)$ is non-trivial, so the topological index of $B^n$ is at most $n$.  
\end{proof}

\begin{cor}
The topological index of $B^n$ is well-defined, and $B^n$ is topologically minimal.
\end{cor}

\subsection{Bounding from below}

We make use of an important theorem in the development of topological index by Bachman:

\begin{theorem}[Theorem 3.7 of \cite{BacTITS3M}]
\label{theorem:BachmanNiceIntersection}
Let $G$ be a properly embedded, incompressible surface in an irreducible 3-manifold $M$. Let $B$ be a properly embedded surface in $M$ with topological index $n$. Then $B$ may be isotoped so that 
\begin{enumerate}
\item $B$ meets $G$ in $p$ saddles, for some $p \leq n$, and
\item the sum of the topological indices of the components of $B \rmv n(G)$, plus $p$ is at most $n$.
\end{enumerate}
\end{theorem}

\begin{proposition}
The surface $B^n$ has topological index no smaller than $n$.
\end{proposition}

\begin{proof} Suppose $S^n$ had topological index $\iota < n$. By Theorem 
\ref{theorem:BachmanNiceIntersection}, $B^n$ can be isotoped to a surface, $B^{n}_0$, so that $B^n_0$ meets $H = n(G)$ in $\sigma$ saddles, the sum of the topological indices of each component of $B^n_0 \rmv n(H)$ is $k$, and $k + \sigma \leq \iota$. Further, we may isotope any annular components of $B^n_0 \rmv H$ that are boundary parallel into $\bd H$ completely into $H$. Observe that this will have no effect on the Euler characteristic of $B^n_0 \rmv H$, nor any effect on the topological index, since such a component will have topological index zero. We consider two different cases. 

First, suppose that there is some component of $B^n_0 \rmv H$ with Euler characteristic less than $-8n$. In this case, because the Euler characteristic of $B^n_0$ is $-6n$, the sum of the Euler characteristics of the remaining components of $B^n_0 \rmv H$ must be greater than $2n$. This implies that there are at least $n + 1$ components of $B^n_0 \rmv H$ with non-negative Euler characteristic. Again, as the sum of the topological indices of each component of $B^n_0 \rmv H$ is $k < n$, there must be at least one component of $B^n_0 \rmv H$ with non-negative Euler characteristic and topological index zero. This is impossible by Theorem \ref{theorem:TopologyBoundsDistanceGraphComplement}.

Second, suppose that the Euler characteristic of each component of $B^n_0 \rmv H$ is bounded below by $-8n$. As the sum of the topological indices of each component of $B^n_0 \rmv H$ is $k < n$, there must be at least one index $j$ so that every component of $B^n_0 \cap M_j$ has topological index zero. Thus, there is a component, $B''$, of $B^n_0 \cap M_j$ which is incompressible and has Euler characteristic bounded below by $-8n$. 

While $B''$ may be boundary compressible, we may boundary compress $B''$ maximally, if necessary, to obtain a surface that is incompressible, boundary incompressible, and not boundary parallel. Since boundary compressions only increase Euler characteristic, the resulting essential surface has Euler characteristic bounded below by $-8n$.  

By Lemma \ref{lemma:DifferenceDistanceCurveComplexArcCurveComplex} and Corollary \ref{cor:ACtoG}, in $M_j$ with $B_j$ a copy of $B'$, we have $d_\mathcal{C}(B_j,\mathcal{L})\leq 2 d_\mathcal{BD}(B_j,\mathcal{L}) \leq 2(1+d_\mathcal{G}(B_j,\Gamma))$.   By Theorem \ref{theorem:TopologyBoundsDistanceGraphComplement}, $d_{\mathcal{G}}(B_j, \Gamma)\leq 2 ( 2 g(B'') + |\bd B''| - 1)$.  By our choice of $\mathcal{L}$ and the fact that $\chi(S)=2-2g(S)-|\bd S|$,  we have $32n+7\leq d_\mathcal{C}(B_j,\mathcal{L})\leq 2+2d_\mathcal{G}(B_j,\Gamma)\leq 8g(B'')+4|\bd B''| -2 =-4\chi(B'')+6$.  On the other hand we have just shown that $-8n\leq \chi(B'')$, a contradiction.  In either case, we find that the topological index of $B^n$ cannot be less than $n$. \end{proof}

\subsection{Hyperbolicity}

We have now shown that $M^n$ contains a surface of topological index $n$.  To prove Theorem \ref{theorem:MainTheorem} it remains to show that $M^n$ is hyperbolic.  

\begin{proposition}
\label{prop:Hyperbolic}
For $n > 1$, $M^n$ is hyperbolic.
\end{proposition}


\begin{proof}
Consider an essential surface $S$ in $M^n$ with Euler characteristic bounded below by zero, chosen to intersect $G$ minimally.  
If $S\cup G=\emptyset$, we arrive at a contradiction to Theorem \ref{theorem:TopologyBoundsDistanceGraphComplement} as $S$ would lie in one of the copies of $M'$. If $S\cup G\neq \emptyset$, the incompressibility and boundary incompressibility of $G$ guarantees that the curves of $S\cup G$ are essential in $S$.  Thus $S\cap M'_i$ is a collection of one or more planar surfaces for some $i$.   This again contradicts Theorem   \ref{theorem:TopologyBoundsDistanceGraphComplement}. Thus, in particular, $M^n$ is prime and atoridal for all $n$.  
Then, as $G$ is an incompressible surface in $M^n$, we conclude that $M^n$ is hyperbolic.
\end{proof}

%
%


\bibliographystyle{plain} 
\bibliography{HyperbolicManifoldsContainingHighTopologicalIndexSurfaces}

\end{document}